\documentclass[11pt]{amsart}
\usepackage{amssymb}
\usepackage{amsthm,amsfonts}

\usepackage{amsfonts}
\usepackage{amsmath}
\usepackage[abbrev]{amsrefs}
\usepackage{enumerate}

\newtheorem{Theorem}{Theorem}[section]

\newtheorem{Lemma}[Theorem]{Lemma}
\newtheorem{Corollary}[Theorem]{Corollary}
\newtheorem{Proposition}[Theorem]{Proposition}

\theoremstyle{definition}

\newcommand\norm[1]{\left\lVert#1\right\rVert}

\begin{document}

\title{Diameter of weak neighborhoods and the Radon-Nikod\'ym property  in Orlicz-Lorentz spaces}

\author{Anna Kami\'nska}
\address{Department of Mathematical Sciences
The University of Memphis, TN 38152-3240}
\email{kaminska@memphis.edu}

\author{Hyung-Joon Tag}
\address{Department of Mathematical Sciences
The University of Memphis, TN 38152-3240}
\email{htag@memphis.edu}

\date{\today}

\begin{abstract}
Given an Orlicz $N$-function $\varphi$ and a positive decreasing weight $w$, we present criteria of the diameter two property and of the Radon-Nikod\'ym property in Orlicz-Lorentz function and sequence spaces $\Lambda_{\varphi,w}$ and $\lambda_{\varphi,w}$.  We show that in the spaces $\Lambda_{\varphi,w}$ or $\lambda_{\varphi,w}$ equipped with the Luxemburg norm, the diameter of any relatively weakly subset of the unit ball in these spaces is  two if and only if $\varphi$ does not satisfy the appropriate $\Delta_2$ condition, while they have the Radon-Nikod\'ym property if and only if $\varphi$ satisfies the appropriate $\Delta_2$ condition.

\end{abstract}

\maketitle

\section{Introduction}
 We characterize  the diameter two property in Orlicz-Lorentz function and sequence spaces equipped with the Luxemburg norm.  A Banach  space X has the diameter two property if every nonempty relatively weakly open subset of the unit ball $B_X$ has the diameter two. For general overview on this property, we refer to \cite{ALM}.  It is well known and easy to show that $C[0,1]$  and $L^1[0,1]$  have the diameter two property. It has been also shown  that every infinite-dimensional uniform algebra \cite{NW}, the set  $C(K,X)$ of all continuous functions from a Hausdorff compact topological space $K$ to a Banach space $X$ \cite{GP}, and a symmetric tensor product of $C(K)$ \cite{AB} have this property.   If a Banach space has the Radon-Nikod\'ym property its unit ball contains denting points, so it contains slices of arbitrarily small diameter \cite{B, DU}. Therefore the Radon-Nikod\'ym property can be considered as  nearly opposite from the diameter two property.  The predual of James tree space $B$ does not have both the Radon-Nikod\'ym property and the diameter two property, and the latter is due to the point of continuity property \cite{EW, 2LT}. 
We recall that a Banach space $X$ satisfies the Daugavet property if, for every rank one operator $T: X \rightarrow X$ and for the identity operator $I$, $\|I +T\| = 1 + \|T\|$ holds. The Daugavet property is stronger than the diameter two property, but they are not equivalent. There are natural examples showing this phenomenon. For instance  in \cite{AK} it was proved that the classical interpolation spaces such as $L^1 + L^\infty$ and $L^1 \cap L^\infty$ do not have the Daugavet property but they do have the diameter two property if equipped with the appropriate norms. Similarly, if an Orlicz function  $\varphi$  does not satisfy appropriate $\Delta_2$ condition then the Orlicz space $L_\varphi$ equipped with the Luxemburg norm fails to have the Daugavet property, but it has the diameter two property \cite{AKM}. The investigation of the diameter two property in Orlicz-Lorentz spaces is inspired by that result.

Recent discovery of an explicit description of the K\"othe dual  spaces of Orlicz-Lorentz spaces \cite{KLR} allows us to investigate a number of geometric properties of these spaces. We use this characterization here to consider the Radon-Nikod\'ym property and the diameter two property in Orlicz-Lorentz spaces $\Lambda_{\varphi,w}$  equipped with the Luxemburg norm. In fact, the K\"othe dual of an Orlicz-Lorentz space $\Lambda_{\varphi,w}$ is the space $\mathcal{M}_{\varphi_*,w}^0$  equipped with the Amemiya norm defined by a different modular $P_{\varphi_*,w}$, where $w$ is a locally integrable, positive, decreasing function and $\varphi_*$ is the Legendre-Fenchel conjugate to an Orlicz function $\varphi$.
For more details on the space $\mathcal{M}_{\varphi,w}$ and its application to Orlicz-Lorentz spaces, we refer to \cite{KR}. 

The article consists of two main parts. In section 2, we state and prove the necessary and sufficient conditions for the diameter two property in Orlicz-Lorentz function spaces $\Lambda_{\varphi,w}$, and in section 3, we show the analogous result for the sequence spaces $\lambda_{\varphi,w}$. In both sections we also characterize the Radon-Nikod\'ym property. In fact, we show that if $\varphi$ satisfies the appropriate $\Delta_2$ condition then the spaces are separable dual spaces, so they possess the Radon-Nikod\'ym property which in turn implies that they have slices of arbitrarily small diameter.
We also show that if $\varphi$ does not satisfy the appropriate $\Delta_2$ condition then the spaces have the diameter two property.

Let $L^0 = L^0(I)$ be a set of all $m$-measurable functions $x : I \rightarrow \mathbb{R}$, where $I = [0, \gamma), 0 < \gamma \leq \infty$, or $I = \mathbb{N}$ and where $m$ is the Lebesgue measure on $[0,\gamma)$ or a counting measure on $\mathbb{N}$. A Banach space $(X, \| \cdot \|)$ is called a Banach function lattice over $(I,m)$ if $X \subset L^0$ and if $0 \leq x \leq y$, where $x\in L^0$ and $y \in X$, then $x \in X$ and $\|x\| \leq \|y\|$.  If $I=[0,\gamma)$ then $X$ is called a Banach function space, and if $I=\mathbb{N}$ then  
the set $L^0(\mathbb{N})$ coincides with the space of all infinite real sequences $x = (x(k))_{k=1}^{\infty}$ and  in this case  $X$ is called a Banach sequence space.  

A Banach function lattice $(X, \| \cdot \|)$ is said to have the Fatou property if for any sequence $(x_n) \subset X$, $x \in L^0$, $x_n \uparrow x$, $m$-a.e., and $\sup_n \|x_n\| < \infty$, we have  $x \in X$, and $\|x_n\| \uparrow \|x\|$. 
Let $X_a\subset X$ be a closed subspace consisting of all order continuous elements from $X$. Recall that $x\in X$ is order continuous whenever for any $0\le x_n \le |x|$ with $x_n\downarrow 0$ $m$-a.e. we have that $\|x_n\|\downarrow 0$.  By $X_b$ denote the closure in $X$ of all simple functions from $X$ with supports of finite measure. We always have  $X_a \subset X_b$. 

The K\"{o}the dual space of a Banach function lattice  $X$, denoted by $X^{\prime}$, is the set of all $x \in L^0$ such that $\|x\|_{X^{\prime}}= \sup\{\int_I xy : \|y\| \leq 1\} < \infty$. The space $X^{\prime}$ equipped with the norm $\| \cdot \|_{X^{\prime}}$ is  also a Banach function lattice  on $(I,m)$. The space $X$ has the Fatou property if and only if $X = X''$ \cite{Z}.
A functional $H\in X^*$ is called regular whenever it has an integral representation $H(x) = \int_I xh$ for some $h\in X'$ and all $x\in X$. The collection of all regular functionals on $X$ is denoted by $X^*_r$. 
If $X_a= X_b$ and $X$ has the Fatou property then $(X_a)^*$ is isometrically isomorphic to $X'$ \cite{BS}. In this case $X^* = (X_a)^* \oplus X_a^\perp$ is isometrically isomorphic to $X' \oplus X_s^*$, where $X_s^* = X_a^\perp$ is the set of singular functionals which coincides with the set of $S\in X^*$ such that $S(x) = 0$ for every $x\in X_a$. Consequently, any $F\in X^*$ has a unique decomposition $F= H + S$, where $H \in X_r^*$ and $S\in X_s^*= X_a^\perp$ \cite{Z}. 

 The distribution function $d_x$ of a function $x\in L^0$ is given by $d_x(\lambda)=m\{t \in I : |x(t)| >\lambda\}$, $\lambda \ge 0$, and the decreasing rearrangement of $x$ is defined as $x^*(t) = \inf\{\lambda \ge 0 : d_{x}(\lambda) \leq t\}$ for $t \ge 0$.  In this paper, decreasing means non-increasing.
By this definition $x^*$ always a function defined on the interval $[0,\infty)$ with the values in $[0,\infty]$. In the case when $\gamma<\infty$, we always have $x^*(t) = 0$ for $t\ge \gamma$, so we will treat $x^*$ as a function defined only on the interval $[0,\gamma)$. In the case when $I=\mathbb{N}$, since the function $x^*$ is constant on each interval $[n-1,n)$, $n\in\mathbb{N}$, we identify it with a sequence of its values $(x^*(n-1))_{n=1}^\infty$. In fact, it coincides with  more convenient formula expressed as
$x^*(k)=(x^*(k))_{k=1}^{\infty}$, where $x^*(k) =\inf\{\lambda > 0 : d_x(\lambda) < k\}$,  $ k \in \mathbb{N}$.

We say that $x,y \in L^0$ are equimeasurable, denoted by $x\sim y$, if $d_x = d_y$ on $[0,\infty)$. A Banach lattice $(X, \|\cdot\|)$ is called a rearrangement invariant (r.i) Banach space if  for $x\in X$, $y\in L^0$ with $x\sim y$, we have $y\in X$ and $\|x\| = \|y\|$. Given a r.i. Banach space over $I=[0,\gamma)$ define its fundamental function as $\phi_X(t) = \|\chi_{[0,t)}\|$ if $t\in I$.

Comprehensive information on Banach function lattices and on rearrangement invariant spaces may be found in \cite{BS, KA, KPS, LT2}.

Recall the well known result by Hardy which remains also true for sequences $x_1=(x_1(n))_{n=1}^\infty, x_2=(x_2(n))_{n=1}^\infty $. 

\begin{Lemma} {\bf Hardy's Lemma} \cite[Proposition 3.6] {BS}
\label{lem:hardy}
 Let $x_1$ and $x_2$ be nonnegative Lebesgue measurable functions on $[0, \gamma)$, $0<\gamma\le \infty$,  and suppose $\int_{0}^{t} x_1(s)ds \leq \int_{0}^{t} x_2(s)ds$ for all $t \in [0, \gamma)$. Let $y$ be any nonnegative decreasing function on $[0,\gamma)$. Then $\int_{0}^{\gamma} x_1(s)y(s)ds \leq \int_{0}^{\gamma} x_2(s)y(s)ds$. 
\end{Lemma}

The function $\varphi:\mathbb{R}_+\to \mathbb{R}_+$ is called an Orlicz function if $\varphi$ is convex,  $\varphi(0)=0$ and $\varphi(u)>0$ for $u> 0$.  The $\Delta_2$ condition of an Orlicz function $\varphi$ plays a key role in the theory of Orlicz spaces and their generalizations. There are three versions of this condition that depend on the measure space and on the  particular property under investigation. We say that  $\varphi$ satisfies the $\Delta_2$ (resp., $\Delta_2^\infty$; $\Delta_2^0$)  condition if there exists $K>0$ (resp., there exist  $ K>0$,  $u_0\ge 0$; $K>0, u_0 > 0$) such that $\varphi(2u) \leq K \varphi(u)$ for all $u\geq 0$ (resp., $u\ge u_0$; $0\le u\le u_0$).
 We say that $\varphi$ is an Orlicz $N$-function if $\varphi$ is an Orlicz function with $\lim_{u \rightarrow 0^+} {\varphi(u)}/{u} = 0$ and $\lim_{u \rightarrow \infty} {\varphi(u)}/{u} = \infty$. 
 The complementary function of $\varphi$, denoted by $\varphi_{\ast}$, is defined by $\varphi_{\ast}(u) = \sup\{uv- \varphi(v) : v \geq 0\}, u \in \mathbb{R}_+$. In fact, it is the restriction to $\mathbb{R}_{+}$ of the Legendre-Fenchel conjugate to an Orlicz function $\varphi$ on $\mathbb{R}_{+}$ extended to the entire real line as $\varphi(u) = \infty$ for $u \in (-\infty, 0)$. It follows that the Young's inequality, $uv \le \varphi(u) + \varphi_*(v)$ for all $u,v\in \mathbb{R}_+$, holds.
   
The following lemma describes useful equivalent expressions  of the $\Delta_2$, $\Delta_2^{\infty}$ and $\Delta_2^0$ conditions.

\begin{Lemma} \cite[Theorem 1.13] {Chen}  \label{le:C} 
An Orlicz function $\varphi$ satisfies the $\Delta_2$ 
(resp., $\Delta_2^\infty$; $\Delta_2^0$) condition if and only if there exist $l>1$ and $K>1$ (resp.,  $l>1, K>1$,  $u_0\ge 0$; $l>1, K>1, u_0 > 0$) such that $\varphi(lu) \leq K \varphi(u)$ for all $u\ge 0$ (resp., $u\ge u_0$; $0\le u\le u_0$). 
\end{Lemma}

\section{Function spaces}

In this section we will consider the Orlicz-Lorentz function space  on $I=[0,\gamma)$.   A positive decreasing function $w: I\to (0,\infty)$ is called a weight function whenever it is locally integrable, that is, $W(t) = \int_0^t w < \infty$ for all $t\in I$. We denote $W(\infty) = \int_{0}^{\infty} w$ if $\gamma= \infty$. 
Given an Orlicz function $\varphi$, a weight function $w$ and the Lebesgue measure $m$ on $\mathbb{R}_+$, for any $x\in L^0$, the modular $\rho$ is defined by
\begin{eqnarray*}
\rho(x) = \rho_{\varphi, w}(x)&=&\int_0^{\gamma} \varphi(x^*(t))w(t)dm(t) = \int_I \varphi(x^*)w.
\end{eqnarray*}
The modular $\rho$ is convex and  orthogonally subadditive, that is, for $x,y\in L^0$, $\rho((x+y)/2) \le (\rho(x) + \rho(y))/2$  and if $|x|\wedge |y| =0$ then $\rho(x+y) \le \rho(x) + \rho(y)$, where 
$|x|\wedge |y| = \min\{|x|, |y|\}$. 
 The  Orlicz-Lorentz space $\Lambda_{\varphi, w}$ is the set of all  $x\in L^0$ such that $\rho(\lambda x) < \infty$ for some $\lambda >0$. It is a r.i. Banach space when the space is equipped with the Luxemburg norm, that is,
\[
 \|x\| = \|x\|_{\varphi,w} = \inf \{ \epsilon>0: \rho(x/\epsilon) \le 1\}.
 \]
 
 The space $(\Lambda_{\varphi,w}, \|\cdot\|)$ satisfies the Fatou property.  Moreover \cite{K},  
 \begin{equation}\label{eq:1}
 (\Lambda_{\varphi, w})_a= (\Lambda_{\varphi, w})_b = \{f \in L^0 :  \forall \lambda,  \ \rho(\lambda f) < \infty \}.
 \end{equation}
For  locally integrable $x,y\in L^0$, we write $x\prec y$ if $\int_0^t x^* \le \int_0^t y^* $ for all $t\in I$.  Given an Orlicz function $\varphi$  and a weight function $w$, the modular $P_{\varphi,w}$ is defined by
\[
P(x)=P_{\varphi,w}(x) = \inf\left\{\int_I \varphi\left(|x|/{v}\right) v:  v\prec w, v\ge 0\right\},\ \ \ \ x\in L^0,\\
\]
and the corresponding function space by
\[
\mathcal{M}_{\varphi,w}= \{x\in L^0, \ \ \ P(\lambda x) < \infty\ \ \text{for some} \ \ \ \lambda > 0\}.
\]
The space is equipped with either the Amemiya norm 
\[
\|x\|_\mathcal{M}^0= \|x\|^0_{\mathcal{M}_{\varphi,w}} = \inf_{k>0} \frac{1}{k}(P(kx) + 1) < \infty, 
\]
or the Luxemburg norm  
\[
 \|x\|_\mathcal{M} =  \|x\|_{\mathcal{M}_{\varphi,w}}= \inf\{\epsilon > 0: P(x/\epsilon) \le 1\}.
 \]
Both norms are equivalent \cite{KLR},
$P (x) = P(x^*)$ and the space $\mathcal{M}_{\varphi,w}$ equipped with either norm  is an r.i. Banach  space \cite{KR}. In this paper, $\mathcal{M}_{\varphi,w}$ and $\mathcal{M}^0_{\varphi,w}$ stand for the space $\mathcal{M}_{\varphi,w}$ equipped with the Luxemburg norm and the Amemiya norm, respectively. 

In view of (\ref{eq:1}) and  \cite[Theorem 2.2]{KLR} we get the following result on bounded linear functionals on 
$\Lambda_{\varphi,w}$.

\begin{Theorem} \label{th:KLR}
\label{th:01}

Let  $w$ be a decreasing weight function and $\varphi$ be an Orlicz $N$-function.  Then the K\"othe dual space to Orlicz-Lorentz space $\Lambda _{\varphi ,w}$ is expressed as
\begin{equation*}
\left( \Lambda _{\varphi ,w}\right) ^{\prime }=\mathcal{M}_{\varphi _{\ast
},w}^{0}
\end{equation*}%
with equality of norms. 
Moreover any $F\in (\Lambda_{\varphi,w})^*$ is uniquely represented as $F=H+S$, where $H$ is a regular functional such that for some $h\in \mathcal{M}_{\varphi_*,w}^0$ we have
\[
H(x) = \int_I xh, \ \ \ \ x\in \Lambda_{\varphi,w},
\]
with $\|H\| = \|h\|^0_{\mathcal{M}_{\varphi_*,w}}$, and $S$ is a singular functional such that
\[
S(x) = 0 \ \ \ \text{for all} \ \ \ \ x\in (\Lambda_{\varphi,w})_a.
\]
\end{Theorem}

\begin{Proposition}\label{prop:fund}
Let $\varphi$ be an Orlicz $N$-function. Then the fundamental function $\phi_\mathcal{M}$ of the space $(\mathcal{M}_{\varphi,w}, \|\cdot\|_\mathcal{M})$ is expressed as
\[
\phi_\mathcal{M}(t) = \frac{t}{W(t)}\Big{/}\varphi^{-1}\left(\frac{1}{W(t)}\right), \ \ \ t\in (0,\gamma).
\] 
Consequently,  $lim_{t\to 0+} \phi_\mathcal{M}(t) = 0$.
\end{Proposition}

\begin{proof}
In order to compute the fundamental function $\phi_{\mathcal{M}}$ we will use the level functions discussed in \cite{KLR}.  Let $x = \chi_{(0,a)}$, $0<a<\gamma$. We have that the interval $(0,a)$ is a maximal level interval with respect to $w$ because $a$ is a maximal number such that  $\int_0^s x/W(s) \le a/W(a)$ for all $s\in (0,a)$ \cite{KLR}. Indeed for $s\in (0,a)$, $(\int_0^s x)/W(s)= s/W(s)\le a/W(a)$ since $w$ is decreasing, and $a$ is maximal since if $s>a$ then $(\int_0^s x)/W(s) = a/W(s) < a/W(a)$. Then the level function $x^0(s) = \frac{a w(s)}{W(a)} \chi_{[0,a)}(s)$. Therefore by Theorem 4.7 in \cite{KLR}, 
\[
P(x) = P(x^0)= \int_0^a \varphi\left(\frac{a}{W(a)}\right) w(s)\, ds = \varphi\left(\frac{a}{W(a)}\right) W(a).
\]
Now  it is straightforward to compute that 
\[
\|\chi_{(0,a)}\|_\mathcal{M} = \frac {a}{W(a)} \Big{/}\varphi^{-1}\left(\frac{1}{W(a)}\right).
\]
The function $a\to a/W(a)$ is increasing, so $\lim_{a\to 0} a/W(a) = L$, where $0\le L <\infty$. Moreover $lim_{a\to 0}\varphi^{-1}(1/W(a)) = \infty$. Hence  $\lim_{a\to 0+}\|\chi_{(0,a)}\|_\mathcal{M} = 0$.

\end{proof}

\begin{Theorem}\label{th:RN}
Let $\varphi$ be an Orlicz $N$-function and let $W(\infty) = \infty$ if $\gamma=\infty$. If $\varphi$ satisfies the $\Delta_2$ condition for $\gamma = \infty$, or $\varphi$ satisfies the $\Delta^\infty_2$ condition for $\gamma < \infty$, then $\Lambda_{\varphi,w}$ is a separable dual space. Consequently, $\Lambda_{\varphi,w}$ has the Radon-Nikod\'ym property. 
\end{Theorem}

\begin{proof} If a Banach function lattice $X$ has the Fatou property and $X_a = X_b$ then  $(X_a)^* = X'$ \cite[Corollary 4.2, p. 23]{BS}.  Consider now the space $\mathcal{M}^0_{\varphi,w}$. 
By Theorem \ref{th:KLR} it is a K\"othe dual space of $\Lambda_{\varphi_*, w}$ since $\varphi_{**} = \varphi$.
By the general theory \cite{Z}, any K\"othe dual space must satisfy the Fatou property. Hence $\mathcal{M}^0_{\varphi,w}$ satisfies this property.

By Theorem 5.5 on p. 67 in \cite{BS},  if $X$ is a r.i. Banach space on a non-atomic measure space and $\lim_{t\to 0^+}\phi_X(t) =0$ then $X_a=X_b$. In view of Proposition \ref{prop:fund}, we have that $\lim_{t\to 0+} \phi_{\mathcal{M}}(t) = 0$ for the space $
\mathcal{M}_{\varphi,w}$, hence also for $
\mathcal{M}^0_{\varphi,w}$ because the Luxemburg and Amemiya norms are equivalent \cite{KLR}. Therefore 
$(\mathcal{M}^0_{\varphi,w})_a = (\mathcal{M}^0_{\varphi,w})_b$.

Note now that $\varphi$ is an Orlicz $N$-function if and only if $\varphi_*$ is an Orlicz $N$-function \cite{Chen}. Therefore all above facts remain true if we substitute  $\varphi$ by $\varphi_*$. 
Then by Theorem \ref{th:KLR} and the Fatou property of $\Lambda_{\varphi,w}$ we get $[(\mathcal{M}^0_{\varphi_*,w})_a]^*= (\mathcal{M}^0_{\varphi_*,w})' = (\Lambda_{\varphi,w})''= \Lambda_{\varphi,w}$, and so $\Lambda_{\varphi,w}$ is a dual space. By the appropriate $\Delta_2$ condition, $W(\infty) = \infty$ and  separability of the Lebesgue  measure we get that $\Lambda_{\varphi,w}$ is separable \cite[Theorem 2.4]{K}. By the well known result \cite{DU, L} it must satisfy the Radon-Nikod\'ym property.
\end{proof}

The next corollary follows from Theorem \ref{th:RN} and from the fact that any Banach space with the Radon-Nikod\'ym property must possess  slices of arbitrarily small diameters \cite{B, DU}. 

\begin{Corollary}\label{cor:weak}
Let $\varphi$ be an Orlicz $N$-function and let $W(\infty) = \infty$ if $\gamma=\infty$. If $\varphi$ satisfies the $\Delta_2$ condition for $\gamma = \infty$, or $\varphi$ satisfies the $\Delta^\infty_2$ condition for $\gamma < \infty$, then there are relatively  weakly open subsets of the unit ball $B_{\Lambda_{\varphi,w}}$ with arbitrarily small diameter.
\end{Corollary}

Recall also that $\int_0^\gamma x^* y^* = \sup_{h\sim y} \int_I |x h| $, and thus 
$\int_I |x h|  \le \int_0^\gamma x^* y^* $ for every $h\sim y$. \cite{BS, KPS}

\begin{Lemma}\label{lem:ineq}  For any $x\in L^0$, a decreasing function $0\le y$ on $I$ and a measurable set $A\subset I$, we have
\[
\int_{0}^{\gamma} (x\chi_A)^*y \leq \int_{0}^{m(A)}x^* y.
\]
\end{Lemma}
\begin{proof} Recall that $\int_{0}^{t} h^* = \sup_{m(E) = t} \int_{E} |h| $ for any $h\in L^0$ (pg 64 in \cite{KPS}). For $t \in[0, \gamma)$ we get
\begin{eqnarray*}
\int_{0}^{t} (x\chi_A)^*&=&\sup_{m(E) = t} \int_{E} |x\chi_A|=  \sup_{m(E) = t} \int_{A\cap E} |x| \\
&\leq& \sup_{m(E)=t} \int_{0}^{m(A\cap E)} x^*= \sup_{m(E) = t} \int_{0}^{\gamma} x^* \chi_{[0, m(A \cap E))} \\
&\leq& \sup_{m(E) = t} \int_{0}^{\gamma} x^* \chi_{[0, \min\{m(A), m(E)\})} =\sup_{m(E) = t} \int_{0}^{\gamma} x^* \chi_{[0, m(A))}\chi_{[0, m(E))}\\
&=& \sup_{m(E) = t} \int_{0}^{m(E)} x^* \chi_{[0, m(A))} = \int_{0}^{t} x^* \chi_{[0, m(A))}
\end{eqnarray*}
Thus, by Lemma \ref{lem:hardy}, $\int_{0}^{\gamma} (x\chi_A)^* y \leq \int_{0}^{\gamma} x^* \chi_{[0, m(A))} y$. 
\end{proof}

The next theorem is the main result in this section.

\begin{Theorem}\label{th:func}
Let $w$ be a weight function such that $W(\infty) = \infty$ if $\gamma = \infty$, and let $\varphi$ be an Orlicz $N$-function.  If $\gamma = \infty$ and $\varphi$ does not satisfy the $\Delta_{2}$ condition, or $\gamma< \infty$ and $\varphi$ does not satisfy the $\Delta^{\infty}_{2}$ condition, then the diameter of any nonempty relatively weakly open subset of the unit ball in Orlicz-Lorentz space $\Lambda_{\varphi, w}$ equipped with the Luxemburg norm is equal to two.
\end{Theorem} 

\begin{proof} 

Let $Z$ be a nonempty relatively weakly open subset of the unit ball in Orlicz-Lorentz space $\Lambda_{\varphi, w}$. Then we can find an element $x \in Z$ such that $\norm{x}=1$. We have  $d_x(\lambda) < \infty$ for any $\lambda > 0$. In fact
\[
1 \ge \rho(x) = \int_0^\gamma \varphi(x^*)w 
\ge \int_{\{s\in [0,\gamma): \  x^*(s) > \lambda\}} \varphi(\lambda) w = \varphi(\lambda) \int_0^\beta w,
\]
where $\beta\le \infty$ is such that the intervals $(0,\beta)$ and $\{s\in [0,\gamma): \ x^*(s) >\lambda\}$ have equal measure. By the assumption $W(\infty) = \infty$ if $\gamma = \infty$, we must have $\beta < \infty$ and so $d_x(\lambda) = d_{x^*}(\lambda) = \beta < \infty$. 

Choose $c>0$ and a Lebesgue measurable set $E\subset [0,\gamma)$ with $m(E)>0$ and $|x(t)| \leq c$ on $E$.  From the fact that $d_x(\lambda)<\infty$ for all $\lambda > 0$, if $\gamma =\infty$, we have $m\{t\in I: |x(t)|\le c\} = \infty$. Hence we choose $E$ such that $|x(t) |\le c$ for $t\in E$ and  $m(E) = \infty$.

Suppose $\varphi$ does not satisfy the $\Delta_{2}$ condition when $\gamma=\infty$, or $\varphi$ does not satisfy the $\Delta_{2}^\infty$ condition when $\gamma< \infty$. Then by Lemma \ref{le:C} there exists $(t_n) \subset (0,\infty)$ such that
for all $n\in\mathbb{N}$, 
\begin{equation}\label{ineq:1}
\varphi\left( \left(1+\frac{1}{n}\right)t_n\right)>2^n\varphi(t_n). 
\end{equation}
Assume  without loss of generality that $t_n \uparrow \infty$ when $\gamma < \infty$ and $t_n \uparrow \infty$ or $t_n \downarrow 0$ when $\gamma = \infty$. 

We consider first when  $t_n \uparrow \infty$. Assume that $m(E)<\infty$.   Then we choose a disjoint sequence of measurable sets $E_n \subset E$ such that for  $n\in \mathbb{N}$,  
\begin{equation}\label{eq:2}
\int_{0}^{m( E_n)}w= \frac{1}{2^n \varphi(t_n)}. 
\end{equation}
Indeed, since $\frac{1}{2^n \varphi(t_n)} \le \frac{1}{2^n\varphi(t_1)}$,  $\sum_{n=1}^{\infty} \frac{1}{2^n\varphi(t_n)} <\infty$, and so  $\sum_{n=n_0}^{\infty} \frac{1}{2^n\varphi(t_n)}< \int_{0}^{m(E)}w$ for some $n_0\in\mathbb{N}$. Then we can find a disjoint sequence of measurable sets $E_{n}\subset E$ such that $\int_{0}^{m(E_{n})}w = \frac{1}{2^{n_0+n}\varphi(t_{n_0+n})}$, $n\in \mathbb{N}$. Without loss of generality we can assume further that $n_0=0$. Since $m(E) < \infty$, $m(E_n)\to 0$. 

Now let  $t_n \downarrow 0$ and  $\gamma = \infty$.  We can still choose  a disjoint sequence of measurable sets $(E_n)$ satisfyting   equation (\ref{eq:2}) and $E_n \subset E$, where $m(E) = \infty$.
Indeed, there exists $(I_n)$, an increasing sequence of measurable subsets of $[0,\infty)$ such that $\cup_{n=1}^{\infty} I_n = [0,\infty)$ and $m (I_n) <  \infty$. Thus $E = \cup_{n=1}^{\infty} E\cap I_n$, where $m(E\setminus  I_n) = \infty$.
 By the continuity of $W$, $W(0)= 0$, and $W(m(E \setminus I_n)) = W(\infty)  = \infty$,  there exist $a_n> 0$ such that $\int_0^{a_n} w = \frac{1}{2^n \varphi(t_n)}$, $n\in \mathbb{N}$.  From the fact that 
$m( E\setminus I_n) = \infty$,
 there exists a disjoint sequence $(E_n)$ of measurable sets satisfying (\ref{eq:2}) and such that $E_n\subset E \setminus I_n$ with $m (E_n) = a_n$.  In this case by (\ref{ineq:1}) we have $m(E_n) \to \infty$.  

Define 
 \[ 
  x_n ' = x\chi_{I \setminus E_n} + t_n\chi_{E_n} \ \ \ \text{and}\ \ \  x_n '' = x\chi_{I \setminus E_n} - t_n\chi_{E_n}.
\]
Note first that $x_n ' \rightarrow x$ and $x_n'' \rightarrow x$  $m$-a.e. on $I$ according to the fact that $E_n$ are disjoint. By the  Fatou property of $\Lambda_{\varphi, w}$, $1=\norm{x} \leq \liminf \norm{x_n'}$, and $1=\norm{x} \leq \liminf \norm{x_n''}$. We will show that $\lim_{n \rightarrow \infty} \norm{x_n '}=1$ and $\lim_{n \rightarrow \infty} \norm{x_n ''}=1$. By the orthogonal subadditivity of $\rho$, Lemma \ref{lem:ineq} and (\ref{eq:2}) we get
\begin{eqnarray*}
\rho(x_n ') &=& \int_{0}^{\gamma} \varphi (x \chi_{I \setminus E_n}+ t_n \chi_{E_n})^*w \leq  \int_{0}^{\gamma} \varphi (x \chi_{I \setminus E_n})^*w + \int_{0}^{\gamma} \varphi (t_n \chi_{E_n})^*w \\
&=&  \int_{0}^{\gamma} (\varphi (|x|)\chi_{I \setminus E_n})^* w+ \int_{0}^{m E_n} \varphi(t_n)w 
\leq \int_{0}^{\gamma} \varphi(|x|)^* \chi_{[0,m (I \setminus  E_n))} w+ \int_{0}^{m E_n} \varphi(t_n) w \\
&\leq& \int_{0}^{\gamma} \varphi(x^*) w + \frac{1}{2^n}
= \rho(x) + \frac{1}{2^n}.
\end{eqnarray*}
Hence
 $\limsup _{n\to\infty}\rho (x_n ') \leq \rho(x) \leq 1$. Then for any $\epsilon > 0$ there exists $n_0$ such that for all $n\geq n_0$, $\rho (x_n') \leq 1 + \epsilon$. It follows by the convexity of $\rho$ that for all $n \geq n_0$, $\rho (x_n'/(1 + \epsilon)) \leq 1$.   Therefore  for all $n \geq n_0, \norm{x_n'} \leq 1 + \epsilon$.
This implies   $\limsup_{n\to\infty} \norm{x_n'} \leq 1$ and proves that $\lim_{n \rightarrow \infty} \norm{x_n '}=1$. Analogously, we get that  $\lim_{n \rightarrow \infty} \norm{x_n ''}=1$.

Let $F$ be a bounded linear functional on $\Lambda_{\varphi, w}$. Then $F=H+S$, where $H$ is the integral functional associated to $h \in \mathcal{M}_{\varphi_* , w}$,  and $S$ is a singular functional identically equal to zero on $(\Lambda_{\varphi, w})_a$ by Theorem \ref{th:KLR}.
We claim that $x-x_n '\in (\Lambda_{\varphi, w})_a$. Note that on $E_n \subset E$ the function $|x|$ is bounded by $c$, so $ |(x-x_n ')(t)|=|x(t)-t_n|\chi_{E_n}(t) \leq (|x(t)|+t_n) \chi_{E_n}(t) \leq (c+t_n)\chi_{E_n}(t)$ on $I$. Then for any $\lambda >0$,
\begin{eqnarray*}
\rho(\lambda(x-x_n'))= \int_{0}^{\gamma} \varphi(\lambda |x-t_n|\chi_{E_n})^* w \leq  \int_{0}^{\gamma} \varphi(\lambda(c+ t_n)\chi_{E_n})^*w 
= \varphi(\lambda(c+t_n)) \int_{0}^{m E_n}w< \infty,
\end{eqnarray*}
 which shows the claim. Hence $S(x- x_n') = 0$ and $F(x-x_n ') =  H(x - x_n') = \int_{E_n} xh dm - \int_{E_n} t_n h dm$.

 Assume  first when $t_n\uparrow \infty$. So  $E_n\subset E$ and $m(E_n)\to 0$.  Since $h \in \mathcal{M}_{\varphi_*, w}$, $P_{\varphi_*, w}(\lambda h) < \infty$ for some $\lambda>0$, and if we let $0 \leq v \in L^0$ and $v  \prec w$, we obtain 

\begin{eqnarray*}
|F(x-x_n ')| &\le&
\lambda^{-1}\int_{I}|x\chi_{E_n}|\, |\lambda h| dm + \lambda^{-1} \int_{I} t_n \chi_{E_n} |\lambda h | dm\\
&\le& \lambda^{-1}\int_0^\gamma (x\chi_{E_n})^*\, (\lambda h)^*  + \lambda^{-1} \int_0^\gamma (t_n \chi_{E_n})^* (\lambda h)^*  \ \ \text{(by Lemma \ref{lem:ineq}}) \\
&\le& \lambda^{-1}\int_0^{m(E_n)} x^*\, (\lambda h^*)  + \lambda^{-1} \int_0^{m(E_n)} t_n  (\lambda h^*) \\
&=& \lambda^{-1}\int_0^{m(E_n)} x^*\, \left( \frac{\lambda h^*}{v} \right) v  + \lambda^{-1} \int_0^{m(E_n)} t_n  \left( \frac{\lambda h^*}{v} \right) v \ \ \text{(by Young's inequality)}\\
&\le&\lambda^{-1} \int_0^{m({E_n})} \left\{ \varphi(x^*)v +\varphi_* \left( \frac{\lambda h^*}{v} \right)v\right\}  + \lambda^{-1} \int_0^{m({E_n})} \left\{ \varphi(t_n)v+ \varphi_* \left( \frac{\lambda h^*}{v} \right) v \right\}. 
\end{eqnarray*}
Due to $v \prec w$, $\int_0^{m({E_n})}\varphi(x^*)v \leq \int_0^{m({E_n})}\varphi(x^*)w$ by Lemma \ref{lem:hardy}.  Taking now the infimum of $\int_0^{m(E_n)} \varphi_* \left(\frac{\lambda h^*}{v} \right)v$ over  $v\prec w$ we get

\begin{eqnarray*}
|F(x-x_n ')| &\leq& \lambda^{-1} \left( \int_0^{m({E_n})} \varphi(x^*) v + 2 \inf \left\{ \int_0^{m(E_n)} \varphi_* \left(\frac{\lambda h^*}{v} \right)v: v \prec w \right \} + \varphi(t_n) \int_0^{m(E_n)} v \right)\\
&\leq&  \lambda^{-1} \left( \int_0^{m({E_n})} \varphi(x^*) w + 2 \inf \left\{ \int_0^{m(E_n)} \varphi_* \left(\frac{\lambda h^*}{v} \right)v: v \prec w \right \} + \varphi(t_n) \int_0^{m(E_n)} w \right).
\end{eqnarray*}
We have $\int_0^{m(E_n)} \varphi(t_n)w = 1/2^n \to 0$ by (\ref{eq:2}). Moreover, $\rho(x)$ is finite and $m(E_n) \to 0$, so $\int_0^{m(E_n)} \varphi(x^*)w \rightarrow 0$. Also  $P_{\varphi_*,w}(\lambda h)$ is finite, so there exists $v_1 \prec w$ such that $\int_0^{\gamma} \varphi_* \left(\frac{\lambda h^*}{v_1} \right) v_1 \leq P_{\varphi_*, w}(\lambda h^*) + 1 < \infty$. Hence $\int_0^{m(E_n)} \varphi_* \left(\frac{\lambda h^*}{v_1} \right) v_1 \rightarrow 0$ as $n \rightarrow \infty$,  and so $ \inf \left\{ \int_0^{m(E_n)} \varphi_* \left(\frac{\lambda h^*}{v} \right)v: v \prec w \right \} \leq \int_0^{m(E_n)} \varphi_* \left(\frac{\lambda h^*}{v_1} \right) v_1 \rightarrow 0$. Thus $|F(x- x_n')| \rightarrow 0$ as $n\to\infty$. For $x_n''$ we show the same, so both $(x_n')$ and $(x_n'')$ converge weakly to $x$.

Now consider the second case when $\gamma = \infty$ and $t_n\downarrow 0$. For some $\lambda > 0$,  $P_{\varphi_*, w}(\lambda h) < \infty$. Then there exists  $0 \leq v \in L^0$,  $v  \prec w$ with $\int_I \varphi_*\left({\lambda |h|}/{v}\right) v \le P_{\varphi_*, w}(\lambda h) + 1<\infty $. From the Young's inequality,
\begin{eqnarray*}
|F(x-x_n')| &\le&
\lambda^{-1}\int_{I}|x\chi_{E_n}| |\lambda h|  + \lambda^{-1} \int_{I} t_n \chi_{E_n} |\lambda h | \\
&=& \lambda^{-1}\int_{E_n}|x|\,\left(\frac{|\lambda h|}{v}\right)v  + \lambda^{-1} \int_{E_n} t_n \left(\frac{|\lambda h |}{v}\right) v\\
&\leq& \lambda^{-1}\left(\int_{E_n}\varphi(|x|)v + 2\int_{E_n} \varphi_*\left(\frac{\lambda |v|}{v}\right) v  + \varphi(t_{n}) \int_{E_n} v \right).
\end{eqnarray*}
 By Lemma \ref{lem:hardy} and by $v\prec w$, $\int_{I} \varphi(|x|) v \le \int_I \varphi(x^*) v^* \le \int_I \varphi(x^*) w =\rho(x) < \infty$. Due to the construction of $E_n$ we get $E_n\subset I \setminus I_n$, where the sequence $(I\setminus I_n)$ is decreasing with $m\bigcap(I\setminus I_n)=0$. Therefore as $n\to\infty$,
\[
\int_{E_n} \varphi(|x|) v \le  \int_{I\setminus I_n} \varphi(|x|)v\to 0. 
\]
By the choice of $v$, $\int_I\varphi_*(\lambda|h|/v) v <\infty$. Hence
\[
\int_{E_n} \varphi_*\left(\frac{\lambda |h|}{v}\right) v\le 
\int_{I\setminus I_n} \varphi_*\left(\frac{\lambda |h|}{v}\right) v  \to 0.
\] 
Finally, from the fact that $\int_{E_n} v\le \int_{0}^{m(E_n)} v^*\le \int_{0}^{m(E_n)} w$ and from (\ref{eq:2}) we have 
\[
\varphi(t_{n}) \int_{E_n} v \le \varphi(t_{n})\int_{0}^{m(E_n)} w = 1/2^n \to 0 .
\]
Consequently in both cases we have that $x_n'\to x$ and $x_n''\to x$ weakly.  

Let's compute now the diameter of $Z$. For $n\in\mathbb{N}$,
\begin{eqnarray*}
\norm{x_n ' - x_n ''} &=& 2\norm{t_n \chi_{E_n}}= 
 2\inf\left\{\lambda>0 : \int_{0}^{m (E_n)} \varphi\left(\frac{t_n}{\lambda}\right)w \leq 1\right\}\\
&=& 2\inf\left\{\lambda>0 : \frac{t_n}{\lambda} \leq \varphi^{-1}\left(\frac{1}{W(m (E_n))}\right)\right\}\\
&=& \frac{2t_n}{\varphi^{-1}\left({1}/{W(m (E_n))}\right)} \overset{\text{by (\ref{eq:2})}}{=} \frac{2t_n}{\varphi^{-1}(2^n \varphi(t_n))} \overset{\text{by (\ref{ineq:1})}} 
{\geq} \frac{2n}{n+1}.
\end{eqnarray*}
Hence  $\norm{x_n ' -x_n ''} \rightarrow 2$ as $n \rightarrow \infty$. Taking $f_n ' = \frac{x_n '}{\norm{x_n '}}$ and $f_n '' = \frac{x_n ''}{\norm{x_n ''}}$, for any bounded linear functional $F$ on $\Lambda_{\varphi,w}$ we get
\[
|F(x-f'_n)| 
\leq |F(x-x_n ')| + \left| F\left(x_n ' - \frac{x_n'}{\norm{x_n'}} \right) \right| \\
\leq  |F(x-x_n')| +\left|1-\frac{1}{\norm{x_n'}}\right|\,\norm{F}\norm{x_n'} \rightarrow 0,
\]
as $n\to\infty$, due to $\|x_n'\|\to 1$. Thus $f_n' \rightarrow x$ weakly and similarly  $f_n'' \rightarrow x$ weakly. 

 We also show that $\norm{f_n' - f_n ''} \rightarrow 2$. Indeed,
\begin{eqnarray*}
\norm{f_n' - f_n ''}  &=& \norm{\frac{x_n '}{\norm{x_n '}} - x_n' +x_n ' - \frac{x_n ''}{\norm{x_n ''}}}
\geq \norm{\frac{x_n '}{\norm{x_n '}} - x_n'} - \norm{\frac{x_n ''}{\norm{x_n ''}}-x_n'}\\
&=& \left|\frac{1}{\norm{x_n'}}-1\right|\, \norm{x_n'} - \norm{\frac{x_n''}{\norm{x_n''}}+x_n'' - x_n '' - x_n'} \\
&\ge& \left|\frac{1}{\norm{x_n'}}-1\right|\, \norm{x_n'}  -\left|\frac{1}{\norm{x_n''}}-1\right|\norm{x_n''} + \norm{x_n' -x_n''}.
\end{eqnarray*}
The last expression approaches  2 as $\norm{x_n'}, \norm{x_n''} \rightarrow 1$ and $\norm{x_n' - x_n''} \rightarrow 2$ as $n \rightarrow \infty$. We have constructed two sequences $(f_n'), (f_n'') \subset Z $ whose distance goes to $2$. This shows that the diameter of $Z$ is two.
The proof is finished.
\end{proof}

As a result of Corollary \ref{cor:weak} and Theorem \ref{th:func} we obtain a full characterization of the relatively weak subsets of a unit ball with diameter two in  Orlicz-Lorentz function spaces equipped with the Luxemburg norm. It is a generalization of the analogous theorem for Orlicz spaces \cite[Theorem 2.5]{AKM}.

\begin{Theorem}\label{th:func-main}
Let $w$ be a decreasing weight function such that $W(\infty) = \infty$ if $\gamma = \infty$, and let $\varphi$ be an Orlicz $N$-function.  Then the diameter of any nonempty relatively weakly open subset of the unit ball in Orlicz-Lorentz function space $\Lambda_{\varphi, w}$ equipped with the Luxemburg norm is equal to $2$ if and only if $\varphi$ does not satisfy the $\Delta_{2}$ condition when $\gamma = \infty$, and $\varphi$ does not satisfy the $\Delta^{\infty}_{2}$ condition when $\gamma< \infty$.
\end{Theorem}

As a corollary of Theorems \ref{th:RN} and  \ref{th:func-main} we obtain a characterization of the Radon-Nikod\'ym property.

\begin{Corollary}
Let $\varphi$ be an Orlicz $N$-function and $w$ a decreasing weight function on $I=[0,\gamma)$ such that $W(\infty) = \infty$ if $\gamma = \infty$. Then the Orlicz-Lorentz space $\Lambda_{\varphi,w}$ has the Radon-Nikod\'ym property if and only if $\varphi$ satisfies the $\Delta_2$ condition if $\gamma=\infty$, and $\varphi$ satisfies the $\Delta_2^\infty$ condition if $\gamma < \infty$.

\end{Corollary} 

\section{Sequence spaces}

In this section 
 we consider the Orlicz-Lorentz sequence space $\lambda_{\varphi, w}$, where $\varphi$ is an Orlicz function and $w= (w(k))_{k=1}^\infty$ is a positive decreasing sequence. Let $W(n) =\sum_{k=1}^n w(k)$, $n\in \mathbb{N}$, and $W(\infty) = \sum_{k=1}^\infty w(k)$. As in function spaces, the modular $\alpha$ for $x=(x(k))_{k=1}^\infty$ is defined by 
\[
 \alpha(x) = \alpha_{\varphi, w}(x) = \sum_{k=1}^{\infty} \varphi(x^*(k))w(k),
 \]
 and an Orlicz-Lorentz sequence space by
 \[
   \lambda_{\varphi, w}= \{x  : \alpha(\lambda x) < \infty,\,\, \text{for some} \,\, \lambda >0 \}.
 \]
 The space $\lambda_{\varphi,w}$ satisfies the Fatou property and the modular $\alpha$ is orthogonally subadditive \cite{KR1}. The Luxemburg norm for $\lambda_{\varphi ,w}$ is given as $\|x\| = \|x\|_{\varphi,w} = \inf\{\epsilon > 0: \alpha(x/\epsilon) \le 1\}$. It is well known that  $(\lambda_{\varphi,w})_a = (\lambda_{\varphi,w})_b =  
 \{x: \rho_{\varphi,w}(\delta x) < \infty \ \ \text{ for all} \ \ \delta >0\}$."   
 
 For a sequence $x\in l^0$, the modular $p_{\varphi,w}$ is defined by 
 \[
 p_{\varphi, w}(x) = \inf \left\{\sum_{k=1}^{\infty} \varphi\left(\frac{|x(k)|}{v(k)}\right) v(k) : v \prec w \right\},
 \]
where $v\prec w$ means that $\sum_{k=1}^n v^*(k) \le \sum_{k=1}^n w(k)$, $n\in\mathbb{N}$. Let
\[
 \mathfrak{m}_{\varphi_, w}= \{x  : p_{\varphi,w}(\delta x) < \infty, \,\, \text{for some} \,\, \delta >0\}.
 \]
 The space $ \mathfrak{m}_{\varphi_, w}$ equipped with the norm $\|x\|_{ \mathfrak{m}}^0 =\|x\|_{ \mathfrak{m}_{\varphi_, w}}^0 = \inf_{k>0}\frac{1}{k}(p_{\varphi,w} (kx) + 1)$  will be denoted by $ \mathfrak{m}^0_{\varphi_, w}$. It  is a r.i. Banach space with $p_{\varphi,w}(x) = p_{\varphi,w}(x^*)$ \cite{KR}. 
In view of Theorem 5.2 in \cite{KLR} we get a sequence  analogue of Theorem \ref{th:KLR}.

\begin{Theorem}  \label{th:KLR-seq}
\label{th:02}
Let  $w$ be a decreasing weight sequence and $\varphi$ be an Orlicz $N$-function.  Then the K\"othe dual space to Orlicz-Lorentz space $\lambda _{\varphi ,w}$ is expressed as
\begin{equation*}
\left(\lambda _{\varphi ,w}\right) ^{\prime }=\mathfrak{m}_{\varphi _{\ast},w}^{0}
\end{equation*}
with equality of norms.    Any functional $F\in (\lambda_{\varphi,w})^*$ is uniquely represented as $F=H+S$, where $H$ is a regular functional such that for some $h = (h(k))_{k=1}^\infty\in \mathfrak{m}_{\varphi_*,w}^0$ we have
\[
H(x) = \sum_{k=1}^\infty x(k) h(k), \ \ \ \ x\in \lambda_{\varphi,w},
\]
with $\|H\| = \|h\|^0_{\mathfrak{m}_{\varphi_*,w}}$, and $S$ is a singular functional such that
\[
S(x) = 0 \ \ \ \text{for all} \ \ \ \ x\in (\lambda_{\varphi,w})_a.
\]
\end{Theorem}

\begin{Theorem}\label{th:RN-seq}
Let $\varphi$ be an Orlicz $N$-function and let $w$ be a weight sequence such that  $W(\infty) = \infty$. If $\varphi$ satisfies the $\Delta^0_2$ condition then $\lambda_{\varphi,w}$ is a separable dual space. Consequently $\lambda_{\varphi,w}$ has the Radon-Nikod\'ym property, and there exist relatively weakly open subsets of the unit ball $B_{\lambda_{\varphi,w}}$ with arbitrarily small diameter.
\end{Theorem}
 
\begin{proof} By Proposition 1 in \cite{KR1}, under the assumption of the $\Delta_2^0$ condition of $\varphi$ and $W(\infty) = \infty$,  $(\lambda_{\varphi,w})_a = \lambda_{\varphi,w}$ and the unit vectors $e_n$ form a  boundedly complete basis in  $\lambda_{\varphi,w}$. It follows that $(\lambda_{\varphi,w})_b = \lambda_{\varphi,w}$ and the space has the Fatou property. Then clearly the space is separable. 

We also have by Theorem 5.4 in \cite{BS} that $(\lambda_{\varphi,w})^* = (\lambda_{\varphi,w})'$.  Then  in view of Theorem \ref{th:KLR-seq}, 
$\mathfrak{m}^0_{\varphi,w} = (\lambda_{\varphi_*,w})'$ and thus the space $\mathfrak{m}^0_{\varphi,w}$ has the Fatou property.

Let $x = (x(i))\in (\mathfrak{m}^0_{\varphi,w})_b$. Then $\|\sum_{i=1}^m x(i) e_i - x\|^0_\mathfrak{m} = \| x\chi_{\{m+1, m+2,\dots\}}\|^0_\mathfrak{m}\to 0$ as $m\to \infty$. Hence $x\in (\mathfrak{m}^0_{\varphi,w})_a$  \cite[Proposition 3.2, p. 14]{BS}.  Thus $(\mathfrak{m}^0_{\varphi,w})_a = (\mathfrak{m}^0_{\varphi,w})_b$. Now similarly as in the function case in view of \cite[Corollary 4.2, p. 23]{BS}, $[(\mathfrak{m}^0_{\varphi,w})_a ]^* = (\mathfrak{m}^0_{\varphi,w})'$.
Finally by Theorem \ref{th:KLR-seq}, 
$[(\mathfrak{m}^0_{\varphi_*,w})_a ]^* =(\mathfrak{m}^0_{\varphi_*,w})' = \lambda_{\varphi,w}$, which shows that $\lambda_{\varphi,w}$ is a dual space.

The conclusion of the proof follows like in Theorem \ref{th:RN} and Corollary \ref{cor:weak}.

\end{proof}

Now, we prove the analogous result  to Theorem \ref{th:func} for the sequence spaces. 

\begin{Theorem}\label{th:seq}
Let $w$ be a weight sequence such that $W(\infty)= \infty$. Suppose $\varphi$ is an Orlicz $N$-function and it does not satisfy the $\Delta_2^0$ condition. Then any nonempty relatively weakly open subset of the unit ball in the Orlicz-Lorentz sequence space $\lambda_{\varphi, w}$ equipped with the Luxemburg norm has the diameter two.  
\end{Theorem}

\begin{proof}
Let Z be a weakly open subset of the unit ball in $\lambda_{\varphi, w}$, and let $x \in Z$ such that $\|x\| = 1$. Suppose $\varphi$ does not satisfy  $\Delta_2^0$. Then there exists $(t_n) \subset (0, \infty)$ such that $t_n \rightarrow 0$ and for $n \in \mathbb{N}$,

\begin{equation} \label{eqn:4}
\varphi\left( \left(1+ \frac{1}{n}\right) t_n\right) > 2^n \varphi(t_n).
\end{equation}

We claim that there exists a sequence of subsets $(E_j)\subset \mathbb{N} \setminus \{1,\dots,j\}$ and a subsequence $(n_j) \subset\mathbb{N}$ such that for all $j\in \mathbb{N}$,

\begin{equation} \label{eqn:5}
\frac{1}{2^j} \leq \varphi(t_{n_j}) \sum_{k=1}^{m(E_j)}w(k) \leq \frac{1}{2^{j-2}}.
\end{equation}
  
Indeed, without loss of generality, assume $w(1) \leq 1$. Then, $w(k) \leq 1$ for any $k \in \mathbb{N}$. Since $\varphi(t_n) \rightarrow 0$, there exists the largest natural number $n_1$ such that $\varphi(t_{n_1}) \leq 1$. Then  there exists $k_1 \geq 1$ such that $\frac{1}{2^{k_1}} \leq \varphi(t_{n_1}) \leq \frac{1}{2^{k_1 - 1}}$. Hence for all $j \in \mathbb{N}$, 
\begin{equation} \label{eqn:6}
\frac{W(j)}{2^{k_1}} \leq \varphi(t_{n_1}) W(j) \leq \frac{W(j)}{2^{k_1 -1}}.
\end{equation}
In view of the assumption $W(\infty) = \infty$  we can find $j \in \mathbb{N}$ such that 
\[
W(j) > 2^{k_1-1},
\]
and let 
\[
m_1 = \min \{j \in \mathbb{N} : W(j) > 2^{k_1-1}\}.
\]
By   $w(1)\le 1$ and $k_1 \ge 1$, we have that $m_1 \ge 2$. By definition of $m_1$ we get that $W(m_1 -1) \leq 2^{k_1 - 1}$. But $W(m_1) = W(m_1 -1) + w(m_1) \leq 2^{k_1 - 1} + 1 \leq 2^{k_1}$.  Now by (\ref{eqn:6}),
\begin{equation*} 
\frac{1}{2} = \frac{2^{k_1-1}}{2^{k_1}}\le \frac{W(m_1)}{2^{k_1}} \leq \varphi(t_{n_1})W(m_1) \leq \frac{W(m_1)}{2^{k_1 -1}} \le \frac{2^{k_1}}{2^{k_1 - 1}} = 2.
\end{equation*}
 Finally let $E_1\subset\mathbb{N}\setminus \{1\}$ be such that $m(E_1) = m_1$, and so  we get (\ref{eqn:5}) for $j=1$. 
 
 As a second step choose $n_2 > n_1$ and $k_2 > k_1$ such that $\frac{1}{2^{k_2}} \leq \varphi(t_{n_2}) \leq \frac{1}{2^{k_2 - 1}}$. 
Let 
\[
m_2 = \min \{j\in \mathbb{N}: W(j) > 2^{k_2 - 2}\}.
\]
 Then $2^{k_2 - 2} \le W(m_2)  = W(m_2 - 1) + w(m_2) \leq 2^{k_2 -2} + 1 \leq 2^{k_2 - 1}$. Hence
\begin{equation*}
\frac{1}{2^2} = \frac{2^{k_2-2}}{2^{k_2}}\le \frac{W(m_2)}{2^{k_2}} \leq \varphi(t_{n_2}) W(m_2) \leq \frac{W(m_2)}{2^{k_2 -1}} \le \frac{2^{k_2-1}}{2^{k_2 - 1}} = 1
\end{equation*} 
Thus, there exists $E_2\subset \mathbb{N} \setminus \{1,2\}$  of size $m_2=m(E_2)$  satisfying (\ref{eqn:5}) for $j=2$.  Now proceeding analogously by induction we can find $E_j \subset \mathbb{N} \setminus \{1,\dots,j\}$ and a subsequence $(n_j)$ satisfying (\ref{eqn:5}).

Define now the sequences $(x_j')_{j=1}^{\infty}, (x_j'')_{j=1}^{\infty}$ by $x_j' = x \chi_{\mathbb{N}\setminus E_j} + t_{n_j} \chi_{E_j}$ and $x_j'' = x \chi_{\mathbb{N}\setminus E_j} - t_{n_j} \chi_{E_j}$. By orthogonal subadditivity of $\alpha$ and by (\ref{eqn:5}),

\begin{eqnarray*}
\alpha(x_j') &=& \sum_{k=1}^{\infty} \varphi((x \chi_{\mathbb{N}\setminus E_j} + t_{n_j} \chi_{E_j})^*(k)) w(k)
 \leq  \sum_{k=1}^{\infty} \varphi((x \chi_{\mathbb{N}\setminus E_j})^*(k)) w(k) +\sum_{k=1}^{\infty} \varphi((t_{n_j} \chi_{E_j})^* (k) w(k))\\
&\leq&  \alpha(x) +\sum_{k=1}^{\infty} \varphi(t_{n_j}) \chi_{\{1,\dots,m(E_j)\}} (k) w(k)
\leq  \alpha(x) + \varphi(t_{n_j}) \sum_{k=1}^{m(E_j)} w(k)
\leq \alpha(x) + \frac{1}{2^{j-2}} \leq 1 + \frac{1}{2^{j-2}}.
\end{eqnarray*} 

\noindent Dividing each side of the above inequality by $1 + \frac{1}{2^{j-2}}$ we get by  convexity of the modular $\alpha$, 
\[
\alpha \left((1+{1}/{2^{j-2}})^{-1} x_j' \right) \leq \left(1+{1}/{2^{j-2}}\right)^{-1} \alpha (x_j') \leq 1,
\]
 which implies that $\|x_j'\| \leq 1 + {1}/{2^{j-2}}$.  Also, $\|x_j'\| \geq \|x \chi_{\mathbb{N} \setminus E_j}\| \geq \|x \chi_{\{1,\dots, j\}}\|$. By the Fatou property of $\lambda_{\varphi,w}$ we have $\| x \chi_{\{1,\dots, j \}} \| \rightarrow \|x\|=1$,  and so $\|x_j'\| \rightarrow 1$ as $j \rightarrow \infty$. Similarly $\|x_j''\| \rightarrow 1$.\\

We claim that $x_j' \rightarrow x$ and $x_j'' \rightarrow x$ weakly. Since $m(E_j) < \infty$, $x-x_j' = x \chi_{E_j} -t_{n_j} \chi_{E_j} \in (\lambda_{\varphi,w})_a$. Then by Theorem \ref{th:KLR-seq}, $F(x -x_j') = H(x-x_j')$ for any $F\in (\lambda_{\varphi,w})^*$. 
Let $H$ be generated by $(\eta (k))_{k=1}^{\infty} = \eta \in \mathfrak{m}_{\varphi_*,w}$. Thus  we can write $H(x) = \sum_{k=1}^{\infty} x(k)\eta(k)$ for $x \in \lambda_{\varphi, w}$.  Since $\eta \in \mathfrak{m}_{\varphi_*,w}$, $p_{\varphi_*,w}(\delta \eta) < \infty$ for some $\delta>0$. Let $v $ be a positive sequence such that $v \prec w$ and 
\begin{equation}\label{eqn:7}
\sum_{k=1}^\infty \varphi_*\left(\frac{\delta |\eta(k)|}{v(k)}\right) v(k) \le p_{\varphi_*,w}(\delta \eta) + 1 < \infty.
\end{equation}
Then by Young's inequality
\begin{eqnarray*}
|H(x-x_j')| &=&\left|\sum_{k=1}^{\infty} (x(k) \chi_{E_j}(k) -t_{n_j} \chi_{E_j}(k))\eta(k)\right|\\
&\leq& \sum_{k=1}^{\infty}|x(k) \chi_{E_j}(k)\eta(k)| + \sum_{k=1}^{\infty}|t_{n_j} \chi_{E_j} \eta(k)| \\
&=& \delta^{-1}\left(\sum_{k=1}^{\infty} \frac{|x(k)\delta\eta(k)| v(k)}{v(k)}\chi_{E_j}(k) + \sum_{k=1}^{\infty} \frac{t_{n_j} \delta |\eta(k)|v(k)}{v(k)} \chi_{E_j}(k)\right)\\
&\leq& \delta^{-1}\left(\sum_{k=1}^{\infty} \varphi(|x(k)|)v(k) \chi_{E_j}(k)+ 2\sum_{k=1}^{\infty} \varphi_*\left(\frac{\delta|\eta(k)|}{v(k)}\right) v(k) \chi_{E_j}(k) + \varphi(t_{n_j}) \sum_{k=1}^{\infty} v(k) \chi_{E_j} \right).
\end{eqnarray*}
By $v\prec w$ and in view of Lemma \ref{lem:hardy}, $\sum_{k=1}^\infty \varphi(|x(k)|) v(k) \le \sum_{k=1}^\infty \varphi(x^*(k)) v^*(k) \le \sum_{k=1}^\infty \varphi(x^*(k)) w(k) =\alpha(x) < \infty$. Hence 
\[
\sum_{k=1}^\infty \varphi(|x(k)|) v(k)\chi_{E_j}(k) \le  \sum_{k=j+1}^\infty \varphi(|x(k)|)v(k)\to 0 \ \ \ \text {as}\ \ \ j\to \infty.
\]
In view of $\sum_{k=1}^\infty v(k) \chi_{E_j}(k)\le \sum_{k=1}^{m(E_j)} v^*(k)\le \sum_{k=1}^{m(E_j)} w(k)$ we get by (\ref{eqn:5}), 
\[
\varphi(t_{n_j}) \sum_{k=1}^{\infty} v(k) \chi_{E_j} (k)\le \varphi(t_{n_j})\sum_{k=1}^{m(E_j)} w(k) \le 1/2^{j-2} \to 0 .
\]
We also have by (\ref{eqn:7}),
\[
\sum_{k=1}^\infty \varphi_*\left(\frac{\delta |\eta(k)|}{v(k)}\right) v(k) \chi_{E_j}(k)\le 
\sum_{k=j+1}^\infty \varphi_*\left(\frac{\delta |\eta(k)|}{v(k)}\right) v(k)  \to 0.
\] 
 Thus, $|H(x- x_j')| \rightarrow 0$, which implies that $x_j' \rightarrow x$ weakly. Similarly, $x_j'' \rightarrow x$ weakly.\
 
Now, we show that $\|x_j'- x_j''\| \rightarrow 2$. 
From (\ref{eqn:5}), $\varphi^{-1} (2^j \varphi(t_{n_j})) \geq \varphi^{-1} \left(\frac{1}{W(m(E_j))}\right) \geq \varphi^{-1}(2^{j-2} \varphi(t_{n_j}))$. Due to convexity of $\varphi$,
\begin{eqnarray*}
\|x_j'- x_j''\|& =& \|2 t_{n_j} \chi_{E_j}\| = 2\inf \left\{\delta > 0 : \sum_{k=1}^{\infty} \varphi \left({(t_{n_j}\chi_{E_j})^*}/{\delta} \right) w \leq 1 \right\} \\
&=&  2\inf \left\{\delta > 0 : \sum_{k=1}^{m(E_j)} \varphi \left({t_{n_j}}/{\delta}\right) w \leq 1\right\}
=  2\inf \left\{\delta > 0 : {t_{n_j}}/{\delta} \leq \varphi^{-1} \left({1}/{W(m(E_j))}\right) \right\}\\
&=&  2\inf \left\{\delta > 0 : \frac{t_{n_j}}{\varphi^{-1} \left({1}/{W(m(E_j))}\right)} \leq \delta \right\}
=  \frac{2t_{n_j}}{\varphi^{-1} ({1}/{W(m(E_j))})}\geq \frac{2t_{n_j}}{\varphi^{-1} (2^j \varphi(t_{n_j}))}\\ 
&\geq& \frac{2t_{n_j}}{\varphi^{-1}(2^{n_j} \varphi(t_{n_j}))} \overset{\text{by (\ref{eqn:4})}}{\geq} \frac{2n_j}{n_j+1} \rightarrow 2 \,\,\text{as}\,\, j \rightarrow \infty.  
\end{eqnarray*}
Taking $f_j' = \frac{x_j'}{\|x_j'\|}$ and $f_j'' = \frac{x_j''}{\|x_j''\|}$, $f_j', f_j'' \in B_{\lambda_{\varphi, w}}$. Finally we can show analogously as for function case that $f_j' \rightarrow x$, $f_j'' \rightarrow x$ weakly, and $\|f_j' - f_j''\| \rightarrow 2$ as $j \rightarrow \infty$, and this completes the proof.

\end{proof}

 The following complete characterization  of  Orlicz-Lorentz sequence spaces with the diameter two property (for Orlicz spaces see \cite{AKM}) results from Theorems \ref{th:RN-seq} and \ref{th:seq}.  

\begin{Theorem}\label{th:seq-main}
Let $w$ be a decreasing weight sequence such that $W(\infty) = \infty$, and let $\varphi$ be an Orlicz $N$-function.  Then the diameter of any nonempty relatively weakly open subset of the unit ball in the Orlicz-Lorentz sequence space $\lambda_{\varphi, w}$ equipped with the Luxemburg norm is equal to two if and only if $\varphi$ does not satisfy the $\Delta^0_{2}$ condition. 
\end{Theorem} 

We finish with a criterion on Radon-Nikod\'ym property that follows immediately from Theorems \ref{th:RN-seq} and \ref{th:seq-main}.

\begin{Corollary}
Let $w$ be a decreasing weight sequence such that $W(\infty) = \infty$, and let $\varphi$ be an Orlicz $N$-function. Then the Orlicz-Lorentz sequence space $\lambda_{\varphi,w}$ has the Radon-Nikod\'ym property if and only if $\varphi$ satisfies the $\Delta_2^0$ condition. 
\end{Corollary}

\end{document}